\documentclass{article}
\usepackage{amssymb,amsmath,amsthm}
\usepackage[usenames]{color}

\newtheorem{thm}{Theorem}[section]
\newtheorem{prop}[thm]{Proposition}
\newtheorem{lem}[thm]{Lemma}
\newtheorem{defn}[thm]{Definition}
\newtheorem{cor}[thm]{Corollary}

\newcommand{\be}{\beta}
\newcommand{\de}{\delta}
\newcommand{\ep}{\epsilon}

\newcommand{\Om}{\Omega}

\newcommand{\cA}{\mathcal A}

\newcommand{\cF}{\mathcal F}
\newcommand{\cG}{\mathcal G}
\newcommand{\cL}{\mathcal L}
\newcommand{\cM}{\mathcal M}
\newcommand{\cN}{\mathcal N}
\newcommand{\cS}{\mathcal S}
\newcommand{\cT}{\mathcal T}

\newcommand{\tP}{\widetilde{P}}

\newcommand{\bC}{\mathbb C}
\newcommand{\bE}{\mathbb E}

\newcommand{\bZ}{\mathbb Z}

\newcommand{\lbl}[1]{\label{#1}}

\newcommand{\bcT}{b\cT}

\title{Generating Product Systems}
\author{Nir Avni\footnote{Department of Mathematics, Harvard University, Cambridge MA 02138. Email address: avni.nir@gmail.com}, Benjamin Weiss\footnote{Institute of Mathematics, The Hebrew University of Jerusalem, Jerusalem 91904 Israel. Email address: weiss@math.huji.ac.il}}
\begin{document}
\maketitle{}

\abstract{Generalizing Krieger's finite generation theorem, we give conditions for an ergodic system to be generated by a pair of partitions, each required to be measurable with respect to a given sub-algebra, and also required to have a fixed size.}

\section{Introduction}

Let $(Z,\cF_Z,\mu)$ be a probability space, and let $T:Z\to Z$ be a measurable, probability preserving transformation of $Z$. If $\mu$ is ergodic with respect to $T$, we say that $(T,\cF_Z,\mu,T)$ is an ergodic system.

A partition $P$ of $Z$ is called generating if the sigma-algebra generated by $\bigvee_{-\infty}^{\infty}T^iP$ and the sigma-algebra of null and co-null sets is equal to $\cF_Z$. If $P$ has $a$ parts, then for any $\bigvee T^iP$-measurable partition $Q$, the entropy $h(T,Q)$ is less than or equal to $\log a$. Hence, if $h(T)>\log a$, then there are no generating partitions with $a$ parts. In \cite{Kr}, Krieger proved the following partial converse:

\begin{thm}[Krieger] \lbl{thm:Krieger} Let $(Z,\cF_Z,\mu,T)$ be an ergodic system, and let $a$ be an integer such that $\log a>h(T)$. Then there is a generating partition $P$ of $Z$ that has $a$ parts.
\end{thm}

The case where $h(T)$ is equal to $\log a$ is different; there is a unique ergodic system for which there is a generating partition with $a$ parts (namely, the Bernoulli system).

Krieger's theorem can be thought of as an infinite version of Shannon's coding theorem for noiseless channels. In Shannon's theorem the block size is finite and the error remains positive but converges to  zero as as the size of the block grows. Shannon's theorem for more general channels was also given a zero error version using inifinite codes, as is usual in ergodic theory in a series of papers including \cite{GOD} and \cite{Ki}

In information theory there is a surprising  extension of Shannon's coding theorem to correlated sources due to Slepian and Wolf, \cite{SW}. Our main result is a zero error version of the Slepian--Wolf result for noiseless channels. Combining the techniques of Kieffer with those of this paper would lead to a zero error version of the Slepian--Wolf theorem in the context of noisy channels, but we leave these extensions for the future.

In order to describe the result, consider the following scenario: Suppose that $\cF_X,\cF_Y\subset\cF_Z$ are sub-sigma-algebras and that $\cF_Z=\cF_X\vee \cF_Y$ (without loss of generality, we may assume that $Z=X\times Y$ is a product space, and that $\cF_X,\cF_Y$ are the corresponding sigma-algebras). We seek a pair of partitions $P_X,P_Y$ such that $P_X$ is $\cF_X$-measurable and has $a$ parts, $P_Y$ is $\cF_Y$-measurable and has $b$ parts, and the partition $P_X\vee P_Y$ generates. There are three necessary conditions for such partitions to exist: the first is that the conditional entropy of $T$ with respect to $\cF_Y$, which we denote by $h(T|\cF_Y)$, is less than or equal to $\log a$, the second condition is that $h(T|\cF_X)\leq\log b$, and the third condition is that $h(T)\leq\log a+\log b$. As in Theorem \ref{thm:Krieger}, we show that these conditions are almost sufficient:

\begin{thm} \lbl{thm:main} Let $(Z,\cF_Z,m,T)$ be an ergodic and probability preserving system, let $\cF_X,\cF_Y$ be invariant sub-sigma-algebras of $\cF_Z$, and let $a$ and $b$ be integers. Assume that $\log a>h(T|\cF_Y)$, that $\log b>h(T|\cF_X)$, and that $\log a+\log b>h(T)$. Then there are partitions $P_X$ and $P_Y$ of $Z$ such that
\begin{enumerate}
\item $P_X$ is measurable with respect to $\cF_X$ and has $a$ parts.
\item $P_Y$ is measurable with respect to $\cF_Y$ and has $b$ parts.
\item $P_X\vee P_Y$ generates the sigma-algebra $\cF_X\vee\cF_Y$.
\end{enumerate}
\end{thm}

More generally, it is straightforward to generalize Theorem \ref{thm:main} to the case of $n$ partitions. Namely, assume that there are $n$ sigma-algebras, $\cF_1,\ldots,\cF_n$, and $n$ integers $a_1,\ldots,a_n$. If for every non-empty subset $S\subset\{1,\ldots,n\}$,
\[
h\left(T|\bigvee_{i\notin S}\cF_i\right)< \sum_{i\in S}\log a_i,
\]
then there are partitions $P_1,\ldots,P_n$ such that $P_i$ has $a_i$ parts and is $\cF_i$-measurable, and that $P_1\vee\ldots\vee P_n$ generates.

The proof of Theorem \ref{thm:main} is given in Section \ref{sec:proof}, and consists of two parts. In the first part (Proposition \ref{prop:first.approximation}) we show that there are partitions $P_X,P_Y$ that satisfy the first two requirements, and such that the partition $P_X\vee P_Y$ is close to being a generating partition (in the sense that the sigma-algebra generated by it contains a partition which is close to a fixed generating partition of $Z$). This is done by a procedure of `painting names along Rohlin towers'. In the second part (Proposition \ref{prop:successive.approximation}) we show that one can slightly change these partitions and get partitions that are even closer to being generating. The proof of Proposition \ref{prop:successive.approximation} uses the procedure of `re-painting names along Rohlin towers'. Applying Proposition \ref{prop:successive.approximation} inductively, we get a Cauchy sequence of partitions that converges to a generating partition, which proves Theorem \ref{thm:main}.

The procedures of painting and repainting along Rohlin towers are explained in Section \ref{sec:painting}. Finally, Section \ref{sec:prelim} contains several auxiliary results, the important one is a corollary of the Shannon-McMillan theorem. No novelty is claimed for this material; it is assembled for the reader's convenience. 

\section{Preliminaries} \lbl{sec:prelim}

\begin{defn} Let $(Z,\cF_Z,\mu,T)$ be an ergodic system.
\begin{enumerate}
\item A (Rohlin-) tower is a pair $\cT=(A,M)$, where $A$ is a subset of $Z$, and $M$ is a natural number, such that for any $0<i<M$, the sets $A$ and $T^iA$ are disjoint. 
\item If $\cT=(A,M)$ is a tower, the set $A$ is called the base of the tower, and is denoted by $\bcT$; the number $M$ is called the height of the tower, and is denoted by $|\cT|$.
\item If $\ep$ is a real number, we say that a tower $\cT=(A,M)$ covers more than $\ep$ of $Z$ if $\mu(A\cup TA\cup\ldots\cup T^{M-1}A)>\ep$.
\item If $\cG$ is an invariant sub-sigma-algebra, we say that $\cT$ is $\cG$-measurable if $A$ is $\cG$-measurable.
\end{enumerate}
\end{defn}

\begin{defn} Let $(Z,\cF_Z,\mu,T)$ be an ergodic system, and let $k$ be an integer. A partition of $Z$ into $k$ parts is a function $P:Z\to\{\overline{0},\ldots,\overline{k-1}\}$. Given a partition $P$ into $k$ parts, we denote the number $k$ by $|P|$.
\end{defn}

\begin{defn} Let $(Z,\cF_Z,\mu,T)$ be an ergodic system, let $\cT=(A,M)$ be a tower, and let $P$ be a partition of $Z$. For every $z\in A$, the $(\cT,P)$-name of $z$ is the tuple $(P(z),\ldots,P(T^{M-1}z))$. We denote the $(\cT,P)$-name of $z$ by $\nu_{\cT,P}(z)$.
\end{defn}

For a proof of the following theorem, see \cite{Ru}.

\begin{thm} \lbl{thm:erg.SM.castles} Let $(Z,\cF_Z,\mu,T)$ be an ergodic system, let $f:Z\to\bC$ be a bounded function, and let $P$ be a partition of $Z$. For every $\ep>0$ there is a natural number $N$ such that if $\cT$ is a tower of height greater than $N$ that covers more than $\ep$ of the space, then
\begin{enumerate}
\item The set of points $z$ in the base of $\cT$ that satisfy
\[
\left|\frac{1}{|\cT|}\left(f(z)+f(Tz)+\ldots+f(T^{|\cT|-1})\right)-\int_Zfd\mu\right|<\ep
\]
has measure greater than $(1-\ep)\mu(A)$.
\item There is a collection $\cN$ of $2^{(h(Z,P)+\ep)|\cT|}$ elements in $|P|^{|\cT|}$ such that 
\begin{enumerate}
\item For every $n=(n_i)\in\cN$, the measure of the set of $z$'s in $A$ whose $(\cT,P)$-name is equal to $n$ is between $2^{-(h(Z,P)+\ep)|\cT|}$ and $2^{-(h(Z,P)-\ep)|\cT|}$.
\item The measure of the set of $z$'s in $A$ whose $(\cT,P)$-name is in $\cN$ is greater than $(1-\ep)\mu(A)$.
\end{enumerate}
\end{enumerate}
\end{thm}

\begin{defn} For a partition $P$, define $\cN_P$ to be the set of words of unspecified length with entries in $\{\overline{0},\ldots,\overline{|P|-1}\}$.
\end{defn}

For a set $\Om$, we denote by $2^\Om$ the collection of finite subsets of $\Om$.

\begin{cor} \lbl{cor:SM.two.partitions} Let $(Z,T)$ be an ergodic system, and let $P,Q$ be two partitions of $Z$. Denote the sub-algebra $\bigvee_{i=-\infty}^\infty T^iP$ by $\cF$. For every $\ep>0$, there is an $N$ such that if $\cT=(A,M)$ is a tower of height greater than $N$ that covers more than $\ep$ of the space, then there is a function 
\[
\Phi: \cN_P\to 2^{\cN_Q}
\]
such that
\begin{enumerate}
\item For every $z\in A$, the set $\Phi(\nu_{\cT,P}(z))$ has less than $2^{|\cT|(h(T,Q|\cF)+\ep)}$ elements.
\item The set of $z$'s in $A$ such that $\Phi(\nu_{\cT,P}(z))$ contains the $(\cT,Q)$-name of $z$ has measure greater than $(1-\ep)\mu(A)$.
\end{enumerate}
\end{cor}

\begin{proof} By Theorem \ref{thm:erg.SM.castles}, if $N$ is large enough, then there are collections
\[
\cN\subset|P|^{|\cT|}\quad \cM\subset|P\times Q|^{|\cT|}
\]
such that for any $n\in\cN$, the measure of the set of points whose $(\cT,P)$-name is equal to $n$ is less than $2^{-|\cT|(h(T,P)-\ep)}$, for every $m\in\cM$, the measure of the set of points whose $(\cT,P\times Q)$-name is equal to $m$ is greater than $2^{-|\cT|(h(T,P\times Q)+\ep)}$, and the set of points whose $(\cT,P)$-names are not in $\cN$ or whose $(\cT,P\times Q)$-names are not in $\cM$ has measure less than $\ep\mu(A)$.

Let $\pi_1:|P\times Q|^{|\cT|}\to|P|^{|\cT|}$ and $\pi_2:|P\times Q|^{|\cT|}\to|Q|^{|\cT|}$ be the co-ordinate projections. Let $n\in\cN_P$. If $n\in\cN$ then
\[
2^{-|\cT|(h(T,P)-\ep)}\geq\mu\{z\in A|\nu_{\cT,P\times Q}(z)\in\pi_1^{-1}(n)\cap\cM\}\geq|\pi^{-1}_1(n)\cap\cM|2^{-|\cT|(h(T,P\times Q)+\ep)}
\]
and so $|\pi_1^{-1}(n)\cap\cM|\leq2^{|\cT|(h(T,Q|\cF)+2\ep)}$. Define $\Phi(n)=\pi_2(\pi_1^{-1}(n)\cap\cM)$ if $n\in\cN$, and $\Phi(n)=\emptyset$ if $n\notin\cN$.
\end{proof}

We will use the standard notation for the entropy function: $H(x)=x\log_2(x)+(1-x)\log_2(1-x)$. The following is well known. For a (much more) general statement, see \cite[Theorem 3.9]{A}.
\begin{lem} \lbl{lem:rev.SM.castles} Let $P$ be a partition, let $\cG$ be an invariant sub-sigma-algebra, and let $\cT=(A,M)$ be a $\cG$-measurable tower that covers more than $1-\ep$ of the space. If there is a $\cG$-measurable function $\Phi:A\to2^{\cN_P}$ such that the number of elements in $\Phi(z)$ is less than $2^{h|\cT|}$ and the set of $z\in A$ such that $\nu_{\cT,P}(z)\in\Phi(z)$ has measure greater than $(1-\ep)\mu(A)$, then $h(T,P|\cG)<h+H(\ep)+2\ep|P|$.
\end{lem}

\begin{defn} Let $P,Q$ be partitions of a probability space $(Z,\mu)$. Define
\[
|P\triangle Q|=\mu\{z\in Z | P(z)\neq Q(z)\}.
\]
\end{defn}

\begin{defn} Let $P$ be a partition, let $\cA$ be an algebra of sets, and let $\ep>0$. We write $P\stackrel{\ep}{\subset}\cA$ if there is an $\cA$-measurable partition $Q$ such that $|P\triangle Q|<\ep$. If $Q$ is a partition, we write $P\stackrel{\ep}{\subset}Q$ if $P\stackrel{\ep}{\subset}\cA_Q$, where $\cA_Q$ is the algebra generated by $Q$.
\end{defn}

The following lemma is evident:

\begin{lem} Let $P_1,P_2,P_3$ be partitions of a probability space $Z$. If $P_1\stackrel{\ep}{\subset}P_2$ and $P_2\stackrel{\de}{\subset}P_3$, then $P_1\stackrel{\ep+\de}{\subset}P_3$.
\end{lem}

\begin{prop} \lbl{prop:dependent.partition} Let $(Z,T)$ be an ergodic system, and let $P,Q$ be partitions of $Z$. Assume that $P\stackrel{\de}{\subset}\bigvee_{-\infty}^{\infty} T^iQ$. Then for every $\ep>0$, there is an $K$, such that for every tower $\cT=(A,M)$, which covers more than $\ep$ of the space and for which $M>K$, there is a function $\Phi:\cN_Q\to2^{\cN_P}$ such that
\begin{enumerate}
\item For every $z\in A$, $|\Phi(\nu_{\cT,Q}(z))|<2^{(H(\de)+\de|P|+\ep)M}$.
\item The set of $z\in A$ such that $\nu_{\cT,P}(z)\in\Phi(\nu_{\cT,Q}(z))$ has measure greater than $(1-\ep)\mu(A)$.
\end{enumerate}
\end{prop}

\begin{proof} By assumption, there is a number $N$, and a function $\phi:|Q|^{2N+1}\to|P|$ such that the set
\[
B=\{z\in Z | \phi(Q(T^{-N}z),\ldots,Q(T^Nz))=P(z)\}
\]
has measure greater than $1-\de$. Let $\eta>0$, to be chosen later. By Theorem \ref{thm:erg.SM.castles}, if $\cT=(A,M)$ is a tower that covers more than $\ep$ of the space, and $M$ is sufficiently large, then the set of $z\in A$ such that 
\[
\frac{|\{0\leq i< M | T^iz\in B\}|}{M}>1-\de-\eta/2
\]
has measure greater than $(1-\ep)\mu(A)$. If we require in addition that $N<\eta M/2$ we get that for those $z$'s,
\begin{equation} \lbl{eq:dependent.partition}
\frac{|\{N\leq i\leq M-N | T^iz\in B\}|}{M}>1-\de-\eta.
\end{equation}

Define $\psi:|Q|^{M}\to|P|^M$ by
\[
\psi((q_i))_j=\left\{\begin{matrix}\phi(q_{j-N},\ldots q_{j+N}) && {\textrm{if $N\leq j\leq M-N$}} \\ \overline{0} && {\textrm{else}}\end{matrix}\right. .
\]

For $(q_i)\in|Q|^M$, let $\Phi((q_i))$ be the set of elements in $|P|^M$ whose Hamming distance from $\psi((q_i))$ is less than $(\eta+\de)M$. The size of $\Phi((q_i))$ is less than
\[
{{M}\choose{(\eta+\de)M}}|P|^{(\eta+\de)M} < 2^{(H(\eta+\de)+(\eta+\de)|P|)M},
\]
which is less than $2^{(H(\de)+\ep+\de|P|)M}$ if $\eta$ is taken small enough. Moreover, for every $z$ that satisfies (\ref{eq:dependent.partition}), $\nu_{\cT,P}(z)\in\Phi(\nu_{\cT,Q}(z))$.
\end{proof}

We end this section by the following lemma, which should be thought of a relative version of the claim that if $f$ is a random function from a set of size $a$ to a set of size $b$ then the non-empty fibers of $f$ have size $1+O(a/b)$.

\begin{lem} \lbl{lem:linearity.of.expectation} Let $(Z,\cF_Z,\mu)$ be a probability space, and let $\Om$ be a finite set. Assume that $\phi:Z\to\Om, \Phi:Z\to2^\Om$ are functions such that $|\Phi(z)|<a$ and $\phi(z)\in\Phi(z)$ for every $z\in Z$. The probability that a random function $\psi:\Om\to[b]$ satisfies that
\begin{equation} \lbl{eq:lem.linearity}
\mu\left\{z\in Z| |\psi^{-1}(\psi(\phi(z)))\cap \Phi(z)|<1+\frac{1}{\ep}\left(\frac{a}{b}\right)\right\}>1-\sqrt{\ep}
\end{equation}
is greater than $1-\sqrt{\ep}$.
\end{lem}

\begin{proof} Let $B$ be the complement of the set in (\ref{eq:lem.linearity}). Then
\[
\bE\mu(B)=\int_Z \Pr(z\in B)dz=\int_Z\Pr(|\psi^{-1}(\psi(\phi(z)))\cap(\Phi(z)\setminus\{\phi(z)\})|>a/\ep b)dz.
\]
Since the restriction of $\psi$ to $\Phi(z)\setminus\{\phi(z)\}$ is independent of $\psi(\phi(z))$,
\[
\bE|\psi^{-1}(\psi(\phi(z)))\cap(\Phi(z)\setminus\{\phi(z)\})|=\frac{a}{b},
\]
and hence the integrand is smaller than $\ep$. This means that $\bE\mu(B)<\ep$, which implies the claim.
\end{proof}

\section{Painting and Re-painting} \lbl{sec:painting}

From this point on, fix an ergodic system $(Z,\cF_Z,T)$. All partitions are of $Z$ and all sigma-algebras are contained in $\cF_Z$.

\begin{defn} A partition $P$ is called $\ell$-admissible, if there is no $z\in Z$ such that $P(z)=P(Tz)=\ldots=P(T^{\ell-1}z)=\overline{0}$.
\end{defn}

\begin{defn} An $\ell$-admissible sequence of length $n$ on $a$ symbols is a sequence in $\{\overline{0},\ldots,\overline{a-1}\}^n$ whose first element is $\overline{1}$, and that does not contain a segment of consecutive $\overline{0}$'s of length $\ell$. The space of $\ell$-admissible sequences of length $n$ on $a$ symbols is denoted by $\cA(n,\ell,a)$.
\end{defn}

The following lemma is evident:

\begin{lem} For every $\ep$, if $\ell$ is big enough, then $\lim_{n\to\infty}\frac{1}{n}\log|\cA(n,\ell,a)|>\log a-\ep$.
\end{lem}

This lemma will ensure that we can use the symbols $\overline{0}$ and $\overline{1}$ to mark the base of the tower and still have enough names for encoding.

\begin{defn} Let $(Z,\cF_Z,\mu,T)$ be an ergodic system, and fix a generating partition $\tP$. Let $\cT=(A,M)$ be a tower, let $\ell<|\cT|$ be an integer, and let $a\geq2$ be an integer. 
\begin{enumerate}
\item A $(\cT,\ell,a)$-painting data is a function
\[
\phi:|\tP|^{|\cT|}\to\cA(|\cT|-\ell,\ell,a).
\]
\item A painting data $\phi$ induces a partition $Q$ given by 
\[
z\mapsto\left\{ \begin{matrix} \phi(\nu_{\cT,\tP}(T^{-i} z))_i & T^{-i}z\in A \textrm{ and } 0\leq i\leq|\cT|-\ell-1 \\ \\ \overline{0} & \textrm{else}\end{matrix} \right.
\]
\end{enumerate}
\end{defn}

\begin{lem} \lbl{lem:painting.base.entropy}
\begin{enumerate}
\item For every tower $\cT$, and every $(\cT,\ell,a)$-painting data $\phi$ that generates a partition $Q$, the base of $\cT$ is measurable with respect to $\bigvee_{-\infty}^\infty T^iQ$.
\item Suppose that $h(T)\geq\log a$ and let $\ep>0$. There is $\ep'>0$ such that if $\ell$ is big enough, $\cT$ is a tower that covers more than $1-\ep'$ of the space, and $|\cT|$ is big enough, then the probability that a random $(\cT,\ell,a)$-painting data $\psi$ that induces a partition $Q$ satisfies that $h(T,Q)>\log a-\ep$, is greater than $1-\ep$.
\end{enumerate}
\end{lem}

\begin{proof} The first claim follows, from the fact that $z\in A$ if and only if $P(z)=\overline{1}$ and $P(T^{-1}z)=\ldots=P^(T^{-k}z)=\overline{0}$.

As for the second claim, let $\ep'<\ep$ be such that $2\ep'|\tP|+H(\ep')<\ep$, and assume that $\cT$ covers more than $\ep'$ of the space. By Theorem \ref{thm:erg.SM.castles}, there is a collection $\cN$ of less than $2^{(h(T)+\ep')|\cT|}$ $(\cT,P)$-names such that the set of $z\in A$ for which $\nu_{\cT,P}(z)\in\cN$ has measure greater than $(1-\ep')\mu(A)$. Applying Lemma \ref{lem:linearity.of.expectation} to the functions $\Phi(z)=\cN$ and $\nu_{\cT,P}(z)$, we get that the probability that 
\[
\mu\{z| |\cN\cap\psi^{-1}\psi(\nu_{\cT,P}(z))|<2^{(h(T)-\log a+\ep')|\cT|}\}>1-\ep'
\]
is greater than $1-\ep'$. In this case, since $z\mapsto\cN\cap\psi^{-1}\psi(\nu_{\cT,P}(z))$ is $\bigvee T^iQ$-measurable, Lemma \ref{lem:rev.SM.castles} says that $h(T|\bigvee T^iQ)<h(T)-\log a+2\ep'|\tP|+H(\ep')<h(T)-\log a+\ep$, which implies the claim.
\end{proof}

\begin{defn} Let $(Z,\cF_Z,\mu,T)$ be an ergodic system, and fix a generating partition $\tP$. Let $P$ be a partition of $Z$, let $\cT=(A,M)$ be a tower, let $\ell$ be an integer, and let $\ep>0$.
\begin{enumerate}
\item A $(\cT,P,\ell,\ep)$-repainting data is a function
\[
\phi:|\tP|^{|\cT|}\to\cA(\ep|\cT|,\ell,a).
\]
\item A repainting data induces a partition of $Z$ given by
\[
z\mapsto\left\{ \begin{matrix} \overline{0} & z\notin \cT \\ \\ P(z) & T^{-i}z\in A, \ep|\cT|\leq i\leq |\cT|-2\ell \\ \\ \phi(\nu_{\cT,P}(T^{-i}z))_{i} & T^{-i}z\in A, 0\leq i\leq \ep|\cT| \\ \\ \overline{0} & T^{-i}z\in A, |\cT|-2\ell\leq i\leq|\cT|\end{matrix} \right.
\]
\end{enumerate}
\end{defn}

\begin{lem} \lbl{lem:repainting.close} Let $P$ be a $\ell$-admissible partition, let $\cT$ be a tower that covers more than $1-\eta$ of the space, and let $\phi$ be a $(\cT,P,\ell,\ep)$-repainting data with associated partition $Q$. Then $|P\triangle Q|<\ep+\eta$.
\end{lem}

\begin{lem} \lbl{lem:repainting.base.entropy} Let $P$ be a $\ell$-admissible partition. 
\begin{enumerate}
\item For every tower $\cT=(A,M)$, if $\phi$ is a $(\cT,P,\ell,\ep)$-repainting data with associated partition $Q$, then $A$ is measurable with respect to $\bigvee_{-\infty}^\infty T^iQ$.
\item If $h(T,P)>\log a-\eta$, then if $\cT$ is sufficiently invariant and covers more than $1-\eta$ of the space, then the probability that the partition $Q$ associated with a random repainting data $\psi$ satisfies $h(T,Q)>\log a-\eta$ is greater than $1-\eta$.
\end{enumerate}
\end{lem}

\begin{proof} The first claim follows from the fact that $z\in A$ if and only if $P(z)=\overline{1}$ and $Q(T^{-1}z)=Q(T^{-2}z)=\ldots=Q(T^{-2\ell}z)=\overline{0}$.

As for the second claim, define towers $\cT^{old}=(T^{\ep|\cT|}A,(1-\ep)|\cT|-2\ell)$ and $\cT^{new}=(A,\ep|\cT|)$. By Theorem \ref{thm:erg.SM.castles}, there is a collection $\cN$ of less than $2^{(h(T)+\eta)|\cT^{new}|}$ $(\cT^{new},P)$-names such that the set of $z\in A$ such that $\nu_{\cT^{new},P}(z)\in\cN$ has measure greater than $(1-\eta)\mu(z)$.

Similarly to the proof of Lemma \ref{lem:painting.base.entropy}, there is a function $\Phi:A\to2^{\cN_P}$ such that 
\begin{enumerate}
\item $\Phi(z)$ depends only on the $(\cT^{old},P)$-name of $T^{\ep|\cT|}z$.
\item $|\Phi(z)|<2^{(h(T,P)-\log a+\eta)|\cT^{old}|}$.
\item The set of $z\in A$ such that $\nu_{\cT^{old},P}(T^{\ep|\cT|}z)\in\Phi(z)$ has measure greater than $(1-\ep)\mu(A)$.
\end{enumerate}

By taking the product of $\Phi$ and $\cN$, there is a function $\Psi:A\to2^{\cN_P}$ such that
\begin{enumerate}
\item $\Psi(z)$ depends only on the $(\cT^{old},Q)$-name of $T^{\ep|\cT|}z$.
\item $|\Psi(z)|<2^{(h(T,P)-\log a+\eta)|\cT^{old}|+h(T,P)|\cT^{new}|}$.
\item The set of $z\in A$ such that $\nu_{\cT,P}(z)\in\Psi(z)$ has measure greater than $(1-\ep)\mu(A)$.
\end{enumerate}

By Lemma \ref{lem:linearity.of.expectation}, the probability that
\[
\mu\{z\in A||\Psi(z)\cap\phi^{-1}(\phi(\nu_{\cT,P}(z)))|<2^{(h(T,P)-\log a+\eta)|\cT^{old}|+(h(T,P)-\log a)|\cT^{new}|}\}>(1-\eta)\mu(A)
\]
is greater than $1-\eta$. In this case, since $z\mapsto\Psi(z)\cap\phi^{-1}(\phi(\nu_{\cT,P}(z)))$ is $\bigvee_{-\infty}^\infty T^iQ$-measurable, Lemma \ref{lem:rev.SM.castles} says that $h(T,P|\bigvee T^iQ)<h(T,P)-\log a+\eta$, which proves the claim.
\end{proof}

\section{Proof of Theorem \ref{thm:main}} \lbl{sec:proof}

Let $(Z,\cF_Z=\cF_X\vee\cF_Y,\mu,T)$ be as in Theorem \ref{thm:main}. By Krieger's theorem, there are generating partitions $\tP_X$ and $\tP_Y$ for $\cF_X$ and $\cF_Y$ respectively. Let $\ep_0>0$ be such that  $h(T)<\log a+\log b-\ep_0$.

The first part of the proof of Theorem \ref{thm:main} is to obtain a pair of partitions that are close to generate, in the following sense:

\begin{defn} A pair of partitions $(P_X,P_Y)$ is called $(\ep,\de)$-good if
\begin{enumerate}
\item The partition $P_X$ is $\cF_X$-measurable and has $a$ parts; the partition $P_Y$ is $\cF_Y$-measurable and has $b$ parts.
\item $\tP_X\stackrel{\ep}{\subset}\bigvee_{-\infty}^\infty T^i(P_X\vee P_Y)$, and $\tP_Y\stackrel{\de}{\subset}\bigvee_{-\infty}^\infty T^i(P_X\vee P_Y)$.
\item $h(T,P_Y)>\log b-\ep_0$.
\end{enumerate}

A pair of partitions $(P_X,P_Y)$ is called $(\ep,\de)$-very good if it is $(\ep,\de)$-good and, in addition, $h(T,P_X)> \log a-\ep_0$.
\end{defn}

Note that if an $(\ep,\de)$-good pair exists, then $h(T,\cF_Y)>\log b-\ep_0$. Similarly, if there is an $(\ep,\de)$-very good pair, then also $h(T,\cF_X)>\log a-\ep_0$.

\begin{prop} \lbl{prop:first.approximation} Assume that $h(T,\cF_Y)\geq\log b$ and $h(T,\cF_X)\geq\log a$. Then for every $\ep>0$, there is a pair of partitions $(P_X,P_Y)$ which is $(\ep,\ep)$-very good.
\end{prop}

\begin{proof}
Choose $\eta>0$ such that
\[
h(T|\cF_X)+2\eta<\log b
\]
and
\[
2\eta(\log|\tP_Y|+1)+H(\eta)<\ep_0
\]
and
\[
h(T)<\log a+\log b-H(\eta)-2\eta(\log|\tP_Y|+4). 
\]

Choose $\ell$ such that
\[
\log b-\eta<\lim_{n\to\infty}\frac{1}{n}\log|\cA(n-\ell,\ell,b)|
\]
and
\[
\log a-\eta<\lim_{n\to\infty}\frac{1}{n}\log|\cA(n-\ell,\ell,a)|.
\]
\begin{enumerate}
\item{\it Choosing the Tower $\cS$:}
By Theorem \ref{thm:erg.SM.castles} and Corollary \ref{cor:SM.two.partitions}, if $\cS$ is an invariant enough tower with base $A$ that covers more than half of the space, then there are
\begin{itemize}
\item A collection $\cN\subset\cN_{\tP_Y}$ of size $2^{|\cS|(h(T,\cF_Y)+\eta)}$.
\item A function $\Phi_0$ from $\cN_{\tP_X}$ to subsets of $\cN_{\tP_Y}$ of size $2^{|\cS|(h(T|\cF_X)+\eta)}$.
\end{itemize}
such that the points $z\in A$ for which
\begin{equation}
\nu_{\cS,\tP_Y}(z)\in\cN
\end{equation}
and
\begin{equation}
\nu_{\cS,\tP_Y}(z)\in\Phi_0(\nu_{\cS,\tP_X}(z))
\end{equation}
has measure greater than $(1-\eta)\mu(A)$. Choose such a tower $\cS$ that is measurable with respect to $\cF_Y$, covers $1-\eta$ of the space, and such that $|\cS|$ satisfies
\begin{equation} \lbl{eq:first.approximation.big.|S|.1}
1+\frac{1}{\eta}\left(\frac{2^{|\cS|(h(T,\cF_Y)+\eta)}}{|\cA(|\cS|-\ell,\ell,b)|}\right)<2^{|\cS|(h(T,\cF_Y)-\log b+2\eta)}
\end{equation}
and
\begin{equation} \lbl{eq:first.approximation.big.|S|.2}
\frac{1}{\eta}\left(\frac{2^{|\cS|(h(T|\cF_X)+\eta)}}{|\cA(|\cS|-\ell,\ell,b)|}\right)<1
\end{equation}

\item{\it Choosing the Partition $P_Y$:}
Choose a random function $f:|\tP_Y|^{|\cS|}\to\cA(|\cS|-\ell,\ell,b)$, and let $P_Y$ be the partition obtained by painting $f$ on $\cS$. It is clear that $P_Y$ is $\cF_Y$-measurable. Define a function $\Phi_1$ from $|P_Y|^{|\cS|}$ to subsets of  $|\tP_Y|^{|\cS|}$ by 
\[
\Phi_1(n)=\left\{ \begin{matrix} f^{-1}(n)\cap\cN & |f^{-1}(n)\cap\cN|<2^{|\cS|(h(T,\cF_Y)-\log b+2\eta)} \\ \emptyset & \textrm{else}\end{matrix} \right. ,
\]
and a function $\Psi$ from $|\tP_X\times P_Y|^{|\cS|}$ to $|\tP_X\times\tP_Y|^{|\cS|}$ by
\[
\Psi(n,m)=\left\{ \begin{matrix} f^{-1}(m)\cap\Phi_0(n) & |f^{-1}(m)\cap\Phi_0(n)|=1 \\ \textrm{undefined} & \textrm{else} \end{matrix} \right. .
\]
By Lemma \ref{lem:linearity.of.expectation} and inequalities (\ref{eq:first.approximation.big.|S|.1}) and (\ref{eq:first.approximation.big.|S|.2}), most paintings satisfy that for all but $2\eta$ of the points $z$ in $A$,
\begin{equation}
\nu_{\cS,\tP_Y}(z)\in\Phi_1(\nu_{\cS,P_Y}(z)) ,
\end{equation}
and
\begin{equation}
\nu_{\cS,\tP_X\times\tP_Y}(z)=\Psi(\nu_{\cS,\tP_X\times P_Y}(z)).
\end{equation}

By Lemma \ref{lem:rev.SM.castles}, we get that 
\begin{equation} \lbl{eq:first.approximation.h(T|P_Y)}
h(T,\tP_Y|\bigvee T^i P_Y)<h(T,\cF_Y)-\log b+2\eta+H(\eta)+2\eta\log |\tP_Y|. 
\end{equation}
In particular, $h(T,P_Y)>\log b-2\eta-H(\eta)-2\eta\log |\tP_Y|>\log b-\ep_0$.

\item{\it Choosing the Tower $\cL$:}
By Corollary \ref{cor:SM.two.partitions} and Equation (\ref{eq:first.approximation.h(T|P_Y)}), if $\cL$ is an invariant-enough tower with base $B$, then there is a function $\Phi_2$ from $|P_Y|^{|\cL|}$ to subsets of $|\tP_Y|^{|\cL|}$ of size $2^{|\cL|(h(T,\cF_Y)-\log b+2\eta(\log|\tP_Y|+2)+H(\eta))}$ such that for all but $\eta$-portion of the points of $B$,
\begin{equation}
\nu_{\cL,\tP_Y}(z)\in\Phi_2(\nu_{\cL,P_Y}(z)).
\end{equation}

Also, by making $\cL$ more invariant and using Corollary \ref{cor:SM.two.partitions}, we can assume that there is a function $\Phi_3$ from $|\tP_Y|^{|\cL|}$ to subsets of $|\tP_X|^{|\cL|}$ of size $2^{|\cL|(h(T|\cF_Y)+\eta)}$ such that for all but $\eta$-portion of the points $z$ in $B$,
\begin{equation}
\nu_{\cL,\tP_X}(z)\in\Phi_3(\nu_{\cL,\tP_Y}(z)).
\end{equation}

By composing $\Phi_2$ and $\Phi_3$ we get a function $\Phi_4$ from $|P_Y|^{|\cL|}$ to subsets of $|\tP_X|^{|\cL|}$ of sizes $2^{|\cL|(h(T)-\log b+2\eta(\log|\tP_Y|+3)+H(\eta))}$ such that for all but $2\eta$-portion of the points $z$ in $B$,
\begin{equation}
\nu_{\cL,\tP_X}(z)\in\Phi_3(\nu_{\cL,P_Y}(z)).
\end{equation}

Finally, we require that $\cL$ is $\cF_X$-measurable and $|\cL|$ satisfies
\begin{equation} \lbl{eq:first.approximation.big.|L|}
\log|\cA(|\cL|-\ell,\ell,a)|>|\cL|(\log a-\eta)
\end{equation}

\item{\it Choosing the Partition $P_X$:}
Choose a random function $g:|\tP_X|^{|\cL|}\to\cA(|\cL|-\ell,\ell,a)$, and let $P_X$ be the partition obtained by painting $g$ on the tower $\cL$. It is clear that $P_X$ is $\cF_X$-measurable. Define a function $\Phi_5$ from $|P_X\times P_Y|^{|\cL|}$ to $|\tP_X\times P_Y|^{|\cL|}$ by 
\[
\Phi_5(n,m)=\left\{ \begin{matrix} (g^{-1}(n)\cap\Phi_4(m),m) & |g^{-1}(n)\cap\Phi_4(m)|=1 \\ \textrm{undefined} & \textrm{else}\end{matrix} \right.
\]
By Lemma \ref{lem:linearity.of.expectation}, there is a function $g$ as above such that for all but $3\eta$ of the points in $B$,
\begin{equation}
\nu_{\cL,\tP_X\times P_Y}(z)=\Phi_5(\nu_{\cL,P_X\times P_Y}(z)).
\end{equation}

\item{\it Conclusion of the proof:} For all but $2\eta$ of the points $z\in Z$, $z$ is in the image of both $\cL$ and $\cS$. Assuming this is the case, by looking at the $P_X$-name of $z$ we can find the smallest non-negative number $i$ such that $T^{-i}z\in B$. Denote $w=T^{-i}z$. Similarly, the $P_Y$-name of $z$ determines the minimal non-negative integer $j$ such that $T^{-j}z\in A$. Denote $u=T^{-j}z$. The $P_X\times P_Y$-name of $z$ determines $\nu_{\cL,P_X\times P_Y}(u)$, and hence, for $1-3\eta$ of the points, determines $\nu_{\tP_X\times P_Y}(u)=\Phi_5(\nu_{\cL,P_X\times P_Y}(u))$. Let now $j$ be the smallest non-negative number such that $T^{-j}z\in A$. The tuple $\nu_{\tP_X\times P_Y}(u)$ determines $\nu_{\cS,\tP_X\times P_Y}(T^{-j}z)$ and so for $1-\eta$ of the points determines $\nu_{\cS,\tP_X\times \tP_Y}(T^{-j}z)=\Psi(\nu_{\cS,\tP_X\times P_Y}(T^{-j}z))$. This, clearly, determines $\tP_X(z)$ and $\tP_Y(z)$.

\end{enumerate}
\end{proof}

The second part of the proof of Theorem \ref{thm:main} is to improve the pair $(P_X,P_Y)$ and make it "more generating".

\begin{defn} \lbl{defn:f} Given $\ep,\de>0$, let
\[
f(\ep,\de)=2\frac{H(\ep)+H(\de)+\ep\log|\tP_X|+\de\log|\tP_Y|}{\log a+\log b-\ep_0-h(T)}.
\]
\end{defn}

Note that as $\ep,\de$ tend to $0$, the function $f(\ep,\de)$ tends to $0$.

\begin{prop} \lbl{prop:successive.approximation} If $\ep,\de>0$ are small enough then the following hold:
\begin{enumerate}
\item If $(Q_X,Q_Y)$ is $(\ep,\de)$-good, then for every $\eta>0$ there is a partition $P_X$ such that $(P_X,Q_Y)$ is $(\eta,2\de)$-good, and $|P_X\triangle Q_X|<f(\ep,\de)$. 
\item If $(Q_X,Q_Y)$ is $(\ep,\de)$-very good, then for every $\eta>0$ there is a partition $P_X$ such that $(P_X,Q_Y)$ is $(\eta,2\de)$-very good, and $|P_X\triangle Q_X|<f(\ep,\de)$.
\end{enumerate}
\end{prop}

\begin{proof} We can assume that $\ep,\de$, and $f(\ep,\de)$ are less than $1/2$. There is a number $N$ such that $\tP_Y\stackrel{\de}{\subset}\bigvee_{-N}^NT^i(Q_X\vee Q_Y)$, and $Q_X\stackrel{\eta}{\subset}\bigvee_{-N}^NT^i\tP_X$. We can assume without loss of generality that $4N\eta<\de$ and that
\[
f(\ep,\de)>\frac{H(\ep)+H(\de)+\ep\log|\tP_X|+\de\log|\tP_Y|+\eta}{\log a+\log b-\ep_0-h(T)-2\eta}.
\]. 

Choose $\ell$ such that
\[
\lim_{n\to\infty}\frac{1}{n}\log|\cA(n,\ell,a)|>\log a-\eta
\]
By changing $Q_X$ slightly (and enlarging $\ell$ if needed), we can also assume that $Q_X$ is $\ell$-admissible.

In order to construct the partition $P_X$, we choose a tower $\cT=(A,M)$ which is $\cF_X$-measurable and covers more than $1-\eta$ of the space, and re-paint the first $f(\ep,\de)|\cT|$ levels of it using a random function $\psi:|\tP_X|^{|\cT|}\to \cA(\ell,f(\ep,\de)|\cT|,a)$. We will show that if $\cT$ is taken as sufficiently invariant, then with high probability (on $\psi$), the obtained partition---which we denote by $P_X$---is good. By Lemma \ref{lem:repainting.close}, $|P_X\triangle Q_X|<f(\ep,\de)+\eta$.

Given $\cT$, define the towers $\cT^{old}=(T^{f(\ep,\de)|\cT|}A,(1-f(\ep,\de))|\cT|-2\ell)$ and $\cT^{new}=(A,f(\ep,\de)|\cT|)$. As $|\cT|\to\infty$, the towers $\cT^{old}$ and $\cT^{new}$ become more and more invariant. Note that both cover more than $f(\ep,\de)/2$ of the space if $|\cT|$ is invariant enough, and that $\nu_{\cT^{old},P_X}(z)=\nu_{\cT^{old},Q_X}(z)$ for every $z\in T^{f(\ep,\de)|\cT|}A$.

Applying Proposition \ref{prop:dependent.partition} to the pairs $(\tP_X,Q_X\vee Q_Y)$ and $(\tP_Y,Q_X\vee Q_Y)$, if $\cT$ is invariant enough, then there is a function $\Phi_1:\cN_{Q_X\vee Q_Y}\to2^{\cN_{\tP_X\vee\tP_Y}}$ such that for any $z\in T^{f(\ep,\de)|\cT|}A$, the set $\Phi_1(\nu_{\cT^{old},Q_X\vee Q_Y}(z))$ has at most 
\[
2^{(H(\ep)+H(\de)+\ep|\tP_X|+\de|\tP_Y|+\eta)|\cT|}
\]
elements, and the set of $z\in T^{f(\ep,\de)|\cT|}A$ for which $\nu_{\cT^{old},\tP_X\vee\tP_Y}(z)\in\Phi_1(\nu_{\cT^{old},Q_X\vee Q_Y}(z))$ has measure larger than $(1-\eta)\mu(A)$.

By Corollary \ref{cor:SM.two.partitions} applied to the pair $(Q_Y,\tP_Y)$, if $\cT$ is invariant enough, then there is a function $\Phi_2:\cN_{Q_Y}\to2^{\cN_{\tP_Y}}$ such that for any $z\in A$, the set $\Phi_2(\nu_{\cT^{new},Q_Y}(z))$ has at most
\[
2^{(h(T,\cF_Y)-\log b+\ep_0+\eta)|\cT^{new}|}
\]
elements, and the set of $z\in A$ such that $\nu_{\cT^{new},\tP_Y}(z)\in\Phi_2(\nu_{\cT^{new},Q_Y}(z))$ has measure greater than $(1-\eta)\mu(A)$.

Applying Corollary \ref{cor:SM.two.partitions} to the pair $(\tP_X,\tP_Y)$, if $\cT$ is invariant enough, then there is a function $\Phi_3:\cN_{\tP_Y}\to2^{\cN_{\tP_X}}$ such that for every $z\in A$, the set $\Phi_3(\nu_{\cT^{new},\tP_Y}(z))$ has at most
\[
2^{(h(T|\cF_Y)+\eta)|\cT^{new}|}
\]
elements, and the set of $z\in A$ such that $\nu_{\cT^{new},\tP_X}(z)\in\Phi_2(\nu_{\cT^{new},\tP_Y}(z))$ has measure greater than $(1-\eta)\mu(A)$.

Let $\Phi_4$ be the composition of $\Phi_3$ and $\Phi_2$. Combining $\Phi_1$ and $\Phi_4$, we get a function $\Phi_5:\cN_{Q_X}\times\cN_{Q_Y}\to2^{\cN_{\tP_X}}$ such that for every $z\in A$, the set $\Phi_5(\nu_{\cT^{old},Q_X}(T^{f(\ep,\de)|\cT|}Tz),\nu_{\cT,Q_Y}(z))$ has size at most
\[
2^{(H(\ep)+H(\de)+\ep|\tP_X|+\de|\tP_Y|+f(\ep,\de)(h(T)-\log b+\ep_0)+3\eta)|\cT|},
\]
and the set of $z\in A$ such that
\[
\nu_{\cT,\tP_X}(z)\in\Phi_5(\nu_{\cT^{old},Q_X}(T^{f(\ep,\de)|\cT|}Tz),\nu_{\cT,Q_Y}(z))
\]
has measure greater than $(1-2\eta)\mu(A)$.

Note that if $|\cT|$ is large enough, then
\[
\frac{2^{(H(\ep)+H(\de)+\ep|\tP_X|+\de|\tP_Y|+f(\ep,\de)(h(T)-\log b+\ep_0)+2\eta)|\cT|}}{|\cA(f(\ep,\de)|\cT|,\ell,a)|}<
\]
\[
=2^{(H(\ep)+H(\de)+\ep|\tP_X|+\de|\tP_Y|+f(\ep,\de)(h(T)-\log b-\log a+\ep_0)+3\eta)|\cT|} \to 0
\]
as $|\cT|\to\infty$. Applying Lemma \ref{lem:linearity.of.expectation} to the functions $\Phi_5(\nu_{\cT^{old},Q_X}(T^{f(\ep,\de)|\cT|}Tz),\nu_{\cT,Q_Y}(z)),\nu_{\cT,\tP_X}(z)$, and the random coloring $\psi$, we get that with high probability on $\psi$, the set of points $z\in A$ for which the intersection
\[
\Phi_6(z)=\Phi_5(\nu_{\cT^{old},Q_X}(T^{f(\ep,\de)|\cT|}Tz),\nu_{\cT,Q_Y}(z)) \cap \psi^{-1}\psi(\nu_{\cT,\tP_X}(z))
\]
is a singleton has measure greater than $(1-3\eta)\mu(A)$.

We claim that
\[
\tP_X\stackrel{4\eta}{\subset}\bigvee_{-\infty}^\infty T^i (P_X\vee Q_Y).
\]
Indeed, let $\cG=\bigvee_{-\infty}^\infty T^i (P_X\vee Q_Y)$. By Lemma \ref{lem:repainting.base.entropy}, the set $A$ is $\cG$-measurable, and so is the function $\be:Z\to\bZ$ that sends $z\in Z$ to the largest non-positive integer $i$ such that $T^iz\in A$. The $\cG$-measurable partition
\[
z\mapsto\left\{\begin{matrix}(\Phi_6(T^{\be(z)}z))_{-\be(z)} & {\textrm{if $\be(z)>-M$ and $\Phi_6(z)$ is a singleton}} \\ \overline{0} & {\textrm{else}}\end{matrix}\right. 
\]
coincides with $\tP_X$ for all but $1-4\eta$ of the space.

As for $\tP_Y$, we have that
\[
\tP_Y\stackrel{\de}{\subset}\bigvee_{-N}^{N}T^i{Q_X}\vee\bigvee_{-N}^NT^iQ_Y\stackrel{2N\eta}{\subset}\bigvee_{-2N}^{2N}T^i{\tP_X}\vee\bigvee_{-N}^NT^iQ_Y \stackrel{4N\eta}{\subset}\cG
\]
and so
\[
\tP_Y\stackrel{\de+6N\eta}{\subset}\cG.
\]

Finally, if $h(T,Q_X)>\log a$, then by Lemma \ref{lem:repainting.base.entropy}, if $\cT$ is sufficiently invariant, the entropy $h(T,P_X)$ is greater than $\log a$ with high probability on the random function $\psi$.
\end{proof}

\begin{proof}[Proof of Theorem \ref{thm:main}] If $h(T,\cF_X)<\log a$ and $h(T,\cF_Y)<\log b$, then by Krieger's theorem there are partitions $\tP_X$ and $\tP_Y$ that are $\cF_X$ and $\cF_Y$ measurable, have $a$ and $b$ parts, and generate $\cF_X$ and $\cF_Y$ respectively. We assume in the following that this does not hold.

Choose $\xi_0$ such that Proposition \ref{prop:successive.approximation} applies to the pair $(2\xi_0,2\xi_0)$. Inductively choose a decreasing sequence $\xi_n$ such that both the sums $\sum_n f(\xi_n,\xi_n)$ and $\sum_n f(\xi_{n+1},2\xi_n)$ converge, where $f$ is defined in \ref{defn:f}.

If $h(T,\cF_X)<\log a$ and $h(T,\cF_Y)\geq\log b$, then by Krieger's theorem , there is a generating partition $\tP_X$ to $\cF_X$. We define a sequence of partitions $P_Y^n$ as follows: by Proposition \ref{prop:first.approximation}, there is a pair of partitions $(P_X,P_Y^0)$ that is $(\xi_0,\xi_0)$-good. It follows that $(\tP_X,P_Y^0)$ is also $(0,\xi_0)$-good. Assuming $P_Y^n$ was defined, applying Proposition \ref{prop:successive.approximation}, there is a partition $P_Y^{n+1}$ such that $(\tP_X,P_Y^{n+1})$ is $(0,\xi_{n+1})$-good and $|P_Y^n\triangle P_Y^{n+1}|<f(0,\xi_n)<f(\xi_n,\xi_n)$. Therefore $\sum_n|P_Y^n\triangle P_Y^{n+1}|<\infty$, so there is a limit partition $P_Y^\infty$. it follows that $(\tP_X,P_Y^\infty)$ satisfy the requirements of the theorem. The same proof holds if $h(T,\cF_Y)<\log b$ and $h(T,\cF_X)\geq\log a$.

Assume finally that $h(T,\cF_X)\geq\log a$ and $h(T,\cF_Y)\geq\log b$. We define a sequence of pairs of partitions $(P_X^n,P_Y^n)$ as follows: by Proposition \ref{prop:first.approximation}, there is a pair of partitions $(P_X^0,P_Y^0)$ that is $(\xi_0,\xi_0)$-very good. Assuming $(P_X^n,P_Y^n)$ have been defined and the pair is $(\xi_n,\xi_n)$-very good, applying Proposition \ref{prop:successive.approximation} there is a partition $P_X^{n+1}$ such that $(P_X^{n+1},P_Y^n)$ is $(\xi_{n+1}/2,2\xi_n)$-very good and $|P_X^n\triangle P_X^{n+1}|<f(\xi_n,\xi_n)$. Applying Proposition \ref{prop:successive.approximation} again, there is a partition $P_Y^{n+1}$ such that the pair $(P_X^{n+1},P_Y^{n+1})$ is $(\xi_{n+1},\xi_{n+1})$-very good and $|P_Y^n\triangle P_Y^{n+1}|<f(\xi_{n+1},2\xi_n)$. Since by definition of the $\xi_n$'s, the sums $\sum|P_X^n\triangle P_X^{n+1}|$ and $\sum|P_Y^n\triangle P_Y^{n+1}|$ converge, there are limit partitions $P_X^\infty$ and $P_Y^\infty$. The pair $(P_X^\infty,P_Y^\infty)$ satisfies the requirements of the theorem.
\end{proof}

\end{document}